\documentclass[11pt]{article}
\usepackage{amsmath,amssymb,amsthm,hyperref,theoremref,lipsum,enumitem,lmodern,lineno}

\newcounter{casenum}

\newcommand\mycom[2]{\genfrac{}{}{0pt}{}{#1}{#2}}
\def\re{{\Re e\,}}
\def\im{{\Im m\,}}
\def\dsum{\displaystyle \sum}
\theoremstyle{plain}
\newtheorem{thm}{Theorem}[section]

\newtheorem{lem}{Lemma}[section]
\newtheorem{pro}{Proposition}[section]

\theoremstyle{definition}
\newtheorem{hyp}{Hypothesis}[section]
\newtheorem{rem}{Remark}[section]
\numberwithin{equation}{section}

\title{Oscillatory behavior and equidistribution of signs of Fourier coefficients of cusp forms}
 \author{Mohammed Amin Amri \footnote{ACSA Laboratory, Department of Mathematics, Faculty of Sciences, Mohammed First University, Oujda, Morocco amri.amine.mohammed@gmail.com} }
 \date{October 17, 2017}
 
\begin{document}

\maketitle
\begin{abstract}
In this paper, we discuss questions related to the oscillatory behavior and the equidistribution of signs for certain subfamilies of Fourier coefficients of integral weight newforms with a non-trivial nebentypus as well as Fourier coefficients of eigenforms of half-integral weight reachable by the Shimura correspondence.
\end{abstract}

\section{Introduction}
Let $f$ be a cusp form of positive real weight with multiplier system, and let $a(n)$ be its $n$-th Fourier coefficient. In \cite{KKP,Pr} Knopp, Kohnen and Pribitkin, proved that the sequence $\{a(n)\}_{n\in\mathbb{N}}$, is oscillatory i.e., for each real number $\phi\in [0,\pi)$, either the sequence $\{\re (a(n)e^{-i\phi})\}_{n\in\mathbb{N}}$ changes sign infinitely often or is trivial. Geometrically speaking, this means that no matter how we slice the plane with a straight line going through the origin, there will always be infinitely many terms of $\{a(n)\}_{n\in\mathbb{N}}$ on either side of the line, unless all the terms are on the line itself, this fact motivates the following questions:
\begin{itemize}
\item[$\bullet$] What is the proportion of integers for which the $a(n)$ lies in the same half-plane? 
\item[$\bullet$] If all the $a(n)$ are on a line, what is the proportion of integers for which the $a(n)$ lies in either side of the origin?
\end{itemize}

The latter question was asked in the particular case when $a(n)$ are real in \cite[Section 6]{KKP}. In the case when $f$ is a newform of integral weight without complex multiplication (CM), the celebrate Sato-Tate conjecture suggests that no matter how we slice the plane with a straight line going through the origin, the proportion of primes for which the $a(p)$ lies in the same half-plane is equal to the half of the proportion of primes for which the $a(p)$ are not on the line. The questions that then naturally arise are: whether this is still true for newforms with CM? Can we infer similar results for $a(n)$ when $n$ runs through natural integers? Numerical calculations seem to suggest that the answer is positive. 

Going further in this direction, in the case when $f$ is a cusp form of half-integral weight with real Fourier coefficients contained in the plus space, Kohnen and Bruinier \cite{bruinier} gave the conjecture  
$$ 
\lim_{x\to\infty}\dfrac{\#\{n\le x\; :\; a(n)\gtrless 0\}}{\#\{n\le x \; :\; a(n)\neq 0\}}=\frac{1}{2}.
$$
In \cite{IW2} Inam and Wiese partially verified this conjecture, more precisely they proved that for a fixed square-free integer $t$, the proportion of integers from $\{tn^2\}_{n\in\mathbb{N}}$ on which the $a(tn^2)$ are of the same sign is equal to the half of the proportion of integers from $\{tn^2\}_{n\in\mathbb{N}}$ on which the $a(tn^2)$ are non-zero.

This work was intended as an attempt to answer the questions mentioned above. However, it seems quite difficult to prove any general theorem here and we can only prove results that seem to point into the right direction. One of the motivations of this paper is an earlier work \cite{Amri} of the author, in which he proved that the Fourier coefficients of a newform supported on prime powers have infinitely many ``angular changes". Moreover, he established the ``angular changes" of some subfamilies of Fourier coefficients of holomorphic cusp forms of half-integral weight reachable via the Shimura correspondence, using a robust analytic tool.

The outline of the paper is as follows. In Section \ref{sec:1} we prove a Sato-Tate theorem for CM newforms with non-trivial nebentypus, which is  presumably well-known to experts. But, it seems that this result has not previously appeared in the literature. We also recall the Sato-Tate theorem for newforms without CM due to its importance in the sequel. In Section \ref{sec:2}, we shall prove that for a given newform, the sub-sequences of its Fourier coefficients $\{a(p^\nu)\}_{p,\text{primes}}$ (for a fixed $\nu\in\mathbb{N}$), and $\{a(p^\nu)\}_{\nu\in\mathbb{N}}$ (for a fixed prime $p$) are oscillatory. Moreover, we calculate the proportion of primes and integers, respectively from $\{p^\nu\}_{p,\text{primes}}$, and $\{p^\nu\}_{\nu\in\mathbb{N}}$, on which $a(p^\nu)$ lies in the same half-plane. In Section \ref{sec:3} we study the oscillatory behavior results for the Fourier coefficients of eigenforms of half-integral weight which are accessible via the Shimura correspondence. Indeed, we prove that the sequences  $\{a(tp^2)\}_{p,\text{primes}}$ and $\{a(tp^{2\nu})\}_{\nu\in\mathbb{N}}$, are oscillatory. Furthermore, we calculate the proportion of primes from $\{tp^2\}_{p,\text{primes}}$ on which the $a(tp^2)$ lies in the same half-plane. Some conclusions are drawn in the final section.

\paragraph{\textbf{Notations}}
Throughout the paper, for any $k\ge 2$, $N\ge 1$ and any Dirichlet character $\varepsilon\pmod N$, we
denote by $r_\varepsilon$ the order of $\varepsilon$. We write $S_k^{\mathrm{new}}(N,\varepsilon)$ for the space of newforms of weight $k$ for the group $\Gamma_0(N)$, with nebentypus $\varepsilon$. If $4\mid N$ we write $S_{k+1/2}(N,\varepsilon)$ for the space of half-integral weight cusp forms, when $k=1$, we shall work only with the orthogonal complement (with respect to the Petersson inner product) of the subspace of $S_{k+1/2}(N,\varepsilon)$ spanned by single-variable unary theta functions. The  letter $\mathcal{H}$ stands for the upper half-plane, for $z\in\mathcal{H}$ we set $q:=e^{2\pi iz}$.

Let $\mathbb{P}$ denote the set of all prime numbers. If $S$ is a subset of $\mathbb{P}$,  we denote by  $\delta(S)$ its natural density (if it exists), and we shall denote by $\pi(x)$ the prime-counting function.

\section{Preliminaries}\label{sec:1}
In this section, we mention some results about the equidistribution of eigenvalues of newforms with non-trivial nebentypus, which are crucial for our purpose. In the CM-case (newforms with CM), we provide a full proof, as we are not aware of any appropriate reference.

Let  $f(z)=\sum_{n\ge1}a(n)q^n\in S_k^{\text{new}}(N,\varepsilon)$ be a normalized newform. Fixing a root of unity $\zeta$ such that $\zeta^2\in\text{Im}(\varepsilon)$, let $p$ be a prime number satisfying $\varepsilon(p)=\zeta^2$, it is clear that $\frac{a(p)}{\zeta}$ is real, hence in view of the Ramanujan-Petersson bound, we have 
\begin{equation}
\frac{a(p)}{2 p^{(k-1)/2}\zeta}\in[-1,1].\label{eq1}
\end{equation} 

The distribution of the sequence $\left(\frac{a(p)}{2 p^{(k-1)/2}\zeta}\right)_p$ on $[-1,1]$ as $p$ varies over primes such that $\varepsilon(p)=\zeta^2$, depends on whether $f$ has complex multiplication (of CM-type) or not (not of CM-type), then we have to consider two cases. Let us first look at the CM-case. We shall prove.
\begin{thm}\label{thmCMST}
Let $f\in S_k^{\mathrm{new}}(N,\varepsilon)$ be a normalized newform of CM-type, write
$$
f(z)=\sum_{n\ge1}a(n)q^n\quad z\in\mathcal{H},
$$
for its Fourier expansion at $\infty$. Assume that the order of $\varepsilon$ and the discriminant of the imaginary quadratic field by which $f$ has complex multiplication are coprime. Let $\zeta$ be a root of unity such that $\zeta^2\in\mathrm{Im}(\varepsilon)$. Then the sequence $\left(\frac{a(p)}{2p^{(k-1)/2}\zeta}\right)_p$ is equidistributed in $[-1,1]$ as $p$ varies over primes satisfying $\varepsilon(p)=\zeta^2,$ with respect to the measure $\mu_{\text{CM}}:=\frac{1}{2\pi}\frac{dt}{\sqrt{1-t^2}}+\frac{1}{2}\delta_{0}$, where $\delta_0$ denotes the Dirac measure concentrated at zero. In particular, for any sub-interval $I\subset [-1,1]$, we have 
$$
\lim_{x\to\infty}\dfrac{\#\{p\le x: \varepsilon(p)=\zeta^2,\;\frac{a(p)}{2p^{(k-1)/2}\zeta}\in I\}}{\#\{p\le x\;:\; \varepsilon(p)=\zeta^2\}}=\mu_{\text{CM}}(I)=\dfrac{1}{2\pi}\int_{I}\dfrac{dt}{\sqrt{1-t^2}}+\dfrac{1}{2}\delta_{0}(I).
$$
\end{thm}

Notice that if $\varepsilon$ is trivial we find Deuring's equidistribution theorem \cite{deuring}. Before we proceed to prove the theorem, we first recall some basic facts concerning newforms with complex multiplication, following the exposition given in \cite[Section 2, pp.8-9]{Fite2015}.

Assume that $f$ is of CM-type, by definition (cf. \cite[Definition, pp. 34]{Ribet}) there exists a Dirichlet character $\chi$ such that 
\begin{equation}\label{eq2}
a(p)=\chi(p)a(p),
\end{equation}
for a set of primes $p$ of density $1$. We see that $\chi$ here has to be a quadratic character and then it corresponds to a quadratic imaginary field say $F$, let $d_F$ denote its discriminant. 

According to \cite[Corollary 3.5, Theorem 4.5]{Ribet} the newform $f$ should arise from an algebraic Hecke character $\xi_f$ of $F$ of modulus $\mathfrak{m}$ (integral ideal of $F$), in the sense that the Fourier expansion of $f$ at $\infty$ can be written as
$$
f(z)=\sum_{\mathfrak{a}}\xi_f(\mathfrak{a})q^{\mathcal{N}(\mathfrak{a})},
$$
where $\mathfrak{a}$ runs through all integral ideals of $F$ and $\mathcal{N}=\mathcal{N}_{F/\mathbb{Q}}$ is the norm relative to the extension $F/\mathbb{Q}$. 

Let $K_f$ denote the number field obtained by adjoining to $\mathbb{Q}$ the Fourier coefficients of $f$.  Let $\ell$ be prime.  For each prime ideal $\lambda$  of $K_f$ lying above $\ell$, there is an irreducible 2-dimensional Galois representation
$$
\rho_{f,\lambda} : \mathrm{Gal}(\overline{\mathbb{Q}}/\mathbb{Q})\rightarrow \mathrm{GL}_2(K_{f,\lambda})
$$
which is unramified outside $N\ell$, and satisfies
\begin{equation}\label{eq3}
\det(\rho_{f,\lambda})=\varepsilon\chi_{\ell}^{k-1},
\end{equation}
where $K_{f,\lambda}$ is the completion of $K_{f}$ at $\lambda$ and $\chi_{\ell}$ is the $\ell$-adic cyclotomic character. 

Let $\mathfrak{p}$ be a prime of $F$ lying above a prime $p$ not dividing $N\ell$, and let $\overline{\mathfrak{p}}$ be its conjugate, then $\rho_{f,\lambda}(\mathrm{Frob}_{\mathfrak{p}})$ has a characteristic polynomial 
$$
P_{\mathfrak{p},\rho_{f,\lambda}}(T):=T^2-a(p)T+p^{k-1}\varepsilon(p),
$$ 
which can be factored into
\begin{equation}\label{eq4}
P_{\mathfrak{p},\rho_{f,\lambda}}(T)=(T-\xi_f(\mathfrak{p}))(T-\xi_f(\mathfrak{\overline{p}})).
\end{equation}
Therefore, in view of \eqref{eq3} we have

\begin{equation}\label{eq5}
\xi_f(\mathfrak{\overline{p}})=\varepsilon(\mathcal{N}(\mathfrak{p}))\overline{\xi_f(\mathfrak{p})}.
\end{equation}

Next, consider the commutative group
$$
G=\left\{\left(\begin{array}{cc}
u & 0 \\ 
0 & \zeta \bar{u}
\end{array}\right)\; | \; \zeta\in\text{Im}(\varepsilon), \; u\in\mathbb{C}^*\;\;|u|=1 \right\}.$$
Let $\mu$ denote the Haar measure of $G$ and $X:=\mathrm{conj}(G)$ the set of
its conjugacy classes. Let $S$ be the set of primes of $F$ lying over $N\ell.$ For every
$\mathfrak{p}\notin S$, define the sequence
$$x_{\mathfrak{p}}:=\left(\begin{array}{cc}
\xi_f(\mathfrak{p})/\mathcal{N}(\mathfrak{p})^{(k-1)/2} & 0 \\ 
0 & \varepsilon(\mathcal{N}(\mathfrak{p}))\overline{\xi_f(\mathfrak{p})}/\mathcal{N}(\mathfrak{p})^{(k-1)/2}
\end{array}\right)\in X,$$
with $\mathfrak{p}$ runs over prime ideals of $F$ such that $\text{Frob}_{\mathfrak{p}}=\mathcal{C}$, where $\mathcal{C}$ is a certain conjugacy class of $\mathrm{Gal}(F_{\varepsilon}/F)$, and $\mathrm{Frob}_{\mathfrak{p}}$ denotes the Frobenius element at $\mathfrak{p}$ in the cyclic extension $F_{\varepsilon}/F$, where $F_{\varepsilon}$ denotes the compositum of the field $\mathbb{Q}(\varepsilon)$ generated by the values of the character $\varepsilon$ and $F$. 

At this point we state the following proposition which can be proved by a similar argument to the one in \cite[Proof of Proposition 3.6]{Fite2014}.
\begin{pro}\label{pro1}
For any conjugacy class $\mathcal{C}$ of $\mathrm{Gal}(F_{\varepsilon}/F)$  the sequence $\{x_{\mathfrak{p}}\}_{\mathfrak{p}}$ is $\mu$-equidistributed on $X$.
\end{pro}

Having disposed of this preliminary setup, we can now return to prove Theorem \ref{thmCMST}.
\begin{proof}[Proof of Theorem \ref{thmCMST}]
Let $\zeta$ be a fixed root of unity, such that $\zeta^2\in\mathrm{Im}(\varepsilon)$. Let $p$ be a prime satisfying $\varepsilon(p)=\zeta^2$. We distinguish two cases.

If $p=\mathfrak{p}\overline{\mathfrak{p}}$ splits in $F$, then from \eqref{eq4} we have
\begin{equation}\label{eq6}
a(p)=\xi_{f}(\mathfrak{p})+\xi_{f}(\overline{\mathfrak{p}}).
\end{equation}
Consider the map $\vartheta : X \to[-1,1]$, got by associating a given element of $X$ to its trace divided by $2\zeta$. Altogether from \eqref{eq5} and \eqref{eq6} we see that the conjugacy class of $x_{\mathfrak{p}}$ is mapped to $\frac{a(p)}{2p^{(k-1)/2}\zeta}$ by $\vartheta$. From Proposition \ref{pro1} and taking into account the isomorphism $\mathrm{Gal}(F_{\varepsilon}/F)\cong \mathrm{Im}(\varepsilon)$, it follows that the sequence $\left(\frac{a(p)}{2p^{(k-1)/2}\zeta}\right)_p$ is equidistributed on $[-1,1]$ as $p$ varies over primes that split in $F$ and satisfying $\varepsilon(p)=\zeta^2$ with respect to the measure $\frac{1}{\pi}\frac{dt}{\sqrt{1-t^2}}$, which is the push-forward measure with respect to $\vartheta$ of the Haar measure $\mu$ on the Sato-Tate group $G$.

If $p$ remains inert in $F$, from \eqref{eq2} we have $a(p)=0$, since $(r_{\varepsilon},d_F)=1$ then the fields $F$ and $\mathbb{Q}(\varepsilon)$ are linearly disjoint over $\mathbb{Q}$. Hence, by Chebotarev's density theorem, half of the primes $p$ for which $\varepsilon(p)=\zeta^2$ split in $F$ and the other half are inert in $F$. Consequently the sequence $\frac{a(p)}{2p^{(k-1)/2}\zeta}$ is equidistributed in $[-1,1]$ as $p$ varies over primes satisfying $\varepsilon(p)=\zeta^2$, with respect to the measure $\mu_{\text{CM}}=\frac{1}{2\pi}\frac{dt}{\sqrt{1-t^2}}+\frac{1}{2}\delta_{0}$, where we use the Dirac measure $\delta_0$ to put half the mass at $0$ to account for the inert primes.
\end{proof} 

In the non-CM situation the equidistribution of the sequence $\left(\frac{a(p)}{2p^{(k-1)/2}\zeta}\right)_p$ in $[-1,1]$ as $p$ varies over primes such that $\varepsilon(p)=\zeta^2$, is given by case 3 of \cite[Theorem B]{ST}.
\begin{thm}(Barnet-Lamb, Geraghty, Harris, Taylor)\label{thmST}
Let  $f\in S_k^{\mathrm{new}}(N,\varepsilon)$ be a normalized newform not of CM-type, write
$$
f(z)=\sum_{n\ge1}a(n)q^n\quad z\in\mathcal{H},
$$
for its Fourier expansion at $\infty$. Let $\zeta$ be a root of unity such that $\zeta^2\in\mathrm{Im}(\varepsilon)$. Then the sequence $\left(\frac{a(p)}{2p^{(k-1)/2}\zeta}\right)_p$ is equidistributed in $[-1,1]$ as $p$ varies over primes satisfying $\varepsilon(p)=\zeta^2,$ with respect to the Sato-Tate measure $\mu_{\text{ST}}:=\frac{2}{\pi}\sqrt{1-t^2}dt$. In particular, for any sub-interval $I\subset [-1,1]$ we have 
$$\lim_{x\to\infty}\dfrac{\#\{p\le x: \varepsilon(p)=\zeta^2,\;\frac{a(p)}{2p^{(k-1)/2}\zeta}\in I\}}{\#\{p\le x\;:\; \varepsilon(p)=\zeta^2\}}=\mu_{\text{ST}}(I)=\frac{2}{\pi}\int_{I}\sqrt{1-t^2}dt.$$
\end{thm}
\section{Equidistribution of sign results for integral weight Newforms}\label{sec:2}
In this section, we shall state two of our main results and shall give a proof of them. Throughout this section, we shall stick to the following notations. 

\begin{hyp}\label{hyp}
Let 
$$
f(z)=\sum_{n\ge1}a(n)n^{(k-1)/2}q^n\quad z\in\mathcal{H},
$$
be a normalized newform of integral weight $k\ge 2$ and level $N\ge 1$, with Dirichlet character $\varepsilon\pmod N$. If $f$ has CM by a quadratic imaginary field $F$, with discriminant $d_F$, we suppose that $(d_F,r_\varepsilon)=1$.
\end{hyp}

Consider $\zeta$ a root of unity such that $\zeta^2\in\mathrm{Im}(\varepsilon)$, if $p$ is a prime number satisfying $\varepsilon(p)=\zeta^2$, in view of \eqref{eq1}, we may write the $p$-th Fourier coefficient of $f$ as follows
\begin{equation}\label{eq7}
a(p)=2\zeta \cos\theta_p,
\end{equation}
for a uniquely defined angle $\theta_p\in [0,\pi]$. Notice that Theorem \ref{thmCMST} and \ref{thmST} are equivalent to say that the sequence $\{\theta_p\}_p$  is equidistributed in $[0,\pi]$, when $p$ runs over primes satisfying $\varepsilon(p)=\zeta^2$ with respect to the measure $\mu$, where $\mu=\frac{1}{2\pi}d\theta+\frac{1}{2}\delta_{\pi/2}$ if $f$ has CM, and $\mu=\frac{2}{\pi}\sin^2\theta d\theta$ otherwise.  

For an integer $\nu\ge 1$, and a real $\phi$ belonging to $[0,\pi)$, we let
$$P_{> 0}(\phi,\nu):=\{p\in\mathbb{P}\; :\; \re(a(p^\nu)e^{-i\phi})> 0\},$$

$$P_{< 0}(\phi,\nu):=\{p\in\mathbb{P}\; :\; \re(a(p^\nu)e^{-i\phi})< 0\},$$
and  
$$P_{\neq 0}(\phi,\nu):=\{p\in\mathbb{P}\; :\; \re(a(p^\nu)e^{-i\phi})\neq 0\}.$$ 

Here is our first main theorem.
\begin{thm}\label{thm:1}
Let $f\in S_k^{\mathrm{new}}(N,\varepsilon)$ be a normalized newform of integral weight $k\ge 2$ and level $N\ge 1$, with Dirichlet character $\varepsilon\pmod N$, satisfying Hypothesis \ref{hyp}. Let
$$
f(z)=\sum_{n\ge 1}a(n)n^{(k-1)/2}q^n\quad z\in\mathcal{H},
$$
be its Fourier expansion at $\infty$. Let $\nu$ be a positive odd integer. Then the sequence $\{a(p^\nu)\}_{p\in\mathbb{P}}$ is oscillatory, and for each $\phi\in[0,\pi)$ the sets $P_{>0}(\phi,\nu)$ and $P_{<0}(\phi,\nu)$ have equal positive natural density, that is, both are precisely half of the natural density of the set $P_{\ne 0}(\phi,\nu)$.
\end{thm}

Before proving this theorem, we need the following preliminary lemmas.
\begin{lem}\label{lem:1}
Let $f\in S_k^{\mathrm{new}}(N,\varepsilon)$ be a normalized newform of integral weight $k\ge 2$ and level $N\ge 1$, with Dirichlet character $\varepsilon\pmod N$, let
$$
f(z)=\sum_{n\ge 1}a(n)n^{(k-1)/2}q^n\quad z\in\mathcal{H},
$$
be its Fourier expansion at $\infty$. Let $p$ be a prime number, and $\zeta$ be a root of unity such that $\zeta^2\in \mathrm{Im}(\varepsilon)$. If $\varepsilon(p)=\zeta^2$ then for any positive integer $\nu$, the $p^\nu$--th Fourier coefficient of $f$ is expressible by the trigonometric identity
$$
a(p^\nu)=\frac{\sin((\nu+1)\theta_p)}{\sin\theta_p}\zeta^\nu,
$$
for some $\theta_p\in (0,\pi)$ and in the limiting cases when $\theta_p=0$ and $\theta_p=\pi$ respectively we have $a(p^{\nu})=(\nu+1)\zeta^\nu$ and $a(p^{\nu})=(-1)^\nu(\nu+1)\zeta^\nu$.
\end{lem}
\begin{rem}
It is worth pointing out that if the weight $k$ is even, then the cases when $a(p^{\nu})=(\nu+1)\zeta^\nu$ and $a(p^{\nu})=(-1)^\nu(\nu+1)\zeta^\nu$ can happen for at most finitely many primes $p$ only. In fact, if we denote by $K_f$ the field generated by all the Fourier coefficients of $f$, and pick a prime $p$ satisfying $\varepsilon(p)=\zeta^2$, so that $\theta_p=0$, or $\pi$, then we should have $\sqrt{p}\in K_f$, which can happen for only finitely many primes, because $K_f$ is a number field.   
\end{rem}
\begin{proof}
Since $f$ is a normalized newform, we have the following power series expansion
$$
\sum_{\nu\ge0}a(p^\nu)X^\nu=\dfrac{1}{1-a(p)X+p^{k-1}\varepsilon(p)X^2}.
$$
Setting $X=x\zeta^{-1}$, and write 
$$
1-a(p)\zeta^{-1}x+p^{k-1}x^2=(1-\alpha_px)(1-\beta_px),
$$ 
one sees
$$
\sum_{\nu\ge0}a(p^\nu)\zeta^{-\nu} x^\nu=\dfrac{1}{(\alpha_p-\beta_p)x}\left(\dfrac{1}{1-\alpha_px}-\dfrac{1}{1-\beta_px}\right).
$$
Now, expanding both geometric series, we deduce that the $p^\nu$-th Fourier coefficient of $f$ is
\begin{equation}
a(p^\nu)=\dfrac{\alpha^{\nu+1}_p-\beta^{\nu+1}_p}{\alpha_p-\beta_p}\zeta^{\nu}.\label{eq:7}
\end{equation}
On the other hand, since $\frac{a(p)}{\zeta}\in\mathbb{R}$, then $\beta_p=\overline{\alpha_p}$ and by Deligne's theorem \cite[Theorem 8.2]{Deligne} we have $|\alpha_p|=|\beta_p|=1$. Thus, we may write $\alpha_p=e^{i\theta_p}$ and $\beta_p=e^{-i\theta_p}$ for some $\theta_p\in[0,\pi]$. Inserting this in \eqref{eq:7} we obtain the desired identities. 
\end{proof}
\begin{lem}\label{lem:2}
We make the same assumptions as in Theorem \ref{thm:1}, and let $\zeta$ be a root of unity such that $\zeta^2\in\mathrm{Im}(\varepsilon)$, then $\frac{a(p^\nu)}{\zeta^\nu}$ is real if $\varepsilon(p)=\zeta^2$. With the notation
$$\mathbb{P}_{\gtrless0}(\zeta,\nu):=\left\{p\in\mathbb{P}\;:\; \varepsilon(p)=\zeta^2,\;\frac{a(p^\nu)}{\zeta^\nu}\gtrless0\right\},$$ 
we have
$$
\delta(\mathbb{P}_{>0}(\zeta,\nu))=\left\{
    \begin{array}{ll}
        \frac{1}{2r_{\varepsilon}} & \mbox{if}\;\; f\;\;  \mbox{is not of CM-type}, \\
        \frac{1}{4r_{\varepsilon}} & \mbox{if}\;\; f\;\;  \mbox{is of CM-type},
    \end{array}
\right.
$$
and 
$$
\delta(\mathbb{P}_{<0}(\zeta,\nu))=\left\{
    \begin{array}{ll}
        \frac{1}{2r_{\varepsilon}} & \mbox{if}\;\; f\;\;  \mbox{is not of CM-type}, \\
        \frac{1}{4r_{\varepsilon}} & \mbox{if}\;\; f\;\;  \mbox{is of CM-type}.
    \end{array}
\right.
$$
\end{lem}
\begin{proof}
Let $p$ be a prime number such that $\varepsilon(p)=\zeta^2$. By the previous lemma, the $p^\nu$--th Fourier coefficient $a(p^\nu)$ of $f$ is expressible by the trigonometric identity
$$
a(p^\nu)=\frac{\sin((\nu+1)\theta_p)}{\sin\theta_p}\zeta^\nu,
$$
where $\theta_p\in (0,\pi)$. Since the set $\{p\in\mathbb{P} : \varepsilon(p)=\zeta^2, \theta_p=0\;\text{or}\;\pi\}$ has density zero, we may assume that $\theta_p$ is different from $0$ and $\pi$. 

Therefore, the sign of $\frac{a(p^\nu)}{\zeta^\nu}$ is the same as the sign of $\sin((\nu+1)\theta_p)$, it follows that
$$
p\in \mathbb{P}_{>0}(\zeta,\nu)\Longleftrightarrow\varepsilon(p)=\zeta^2,\;\theta_p\in A_{>0}:=\bigcup_{j=1}^{\frac{\nu+1}{2}}\left(\frac{(2j-2)\pi}{\nu+1},\frac{(2j-1)\pi}{\nu+1}\right),
$$
and
$$
p\in \mathbb{P}_{<0}(\zeta,\nu)\Longleftrightarrow \varepsilon(p)=\zeta^2,\;\theta_p\in A_{<0}:=\bigcup_{j=1}^{\frac{\nu+1}{2}}\left(\frac{(2j-1)\pi}{\nu+1},\frac{2j\pi}{\nu+1}\right).
$$

On the other hand from \cite[Proof of Theorem 1.1, odd case]{Meher2017}, we have
$$
\mu(A_{>0})=\mu(A_{<0})=\left\{
    \begin{array}{ll}
        \frac{1}{2} & \mbox{if}\;\; f\;\;  \mbox{is not of CM-type}, \\
        \frac{1}{4} & \mbox{if}\;\; f\;\;  \mbox{is of CM-type}.
    \end{array}
\right.
$$
Taking into account that
$$
\delta\left(\{p\in\mathbb{P} : \varepsilon(p)=\zeta^2\}\right)=\dfrac{1}{r_{\varepsilon}},
$$
the desired conclusion can be derived easily from Theorem \ref{thmCMST} and \ref{thmST}.
\end{proof}

Now we are in the position to prove Theorem \ref{thm:1}.
\begin{proof}[Proof of Theorem \ref{thm:1}]
\sloppy Fix $\phi\in[0,\pi)$. Let $\zeta$ be a root of unity such that $\zeta^2\in\mathrm{Im}(\varepsilon)$. For the notational convenience throughout the proof we let $\pi_{>0}(x,\zeta):=\#\{p\le x : p\in\mathbb{P}_{>0}(\zeta,\nu)\}$ and similarly $\pi_{<0}(x,\zeta)$, where $\mathbb{P}_{>0}(\zeta,\nu)$ and $\mathbb{P}_{<0}(\zeta,\nu)$ be as in Lemma \ref{lem:2}. 

Let us first examine the oscillatory behavior of the sequence $\{a(p^\nu)\}_{p\in\mathbb{P}}$. We need to consider the following two cases.
\begin{description}
\item[\textbf{Case 1}: $\mathrm{arg}(\zeta^\nu)\not\equiv\phi\pm\frac{\pi}{2}\pmod{2\pi}$.]
The sequence $\{\re(a(p^\nu) e^{-i\phi})\}_{p,\varepsilon(p)=\zeta^2}$ is not trivial. Moreover, by Lemma \ref{lem:1} we have the trigonometric identity
$$
\re(a(p^\nu)e^{-i\phi})=\dfrac{\sin((\nu+1)\theta_p)}{\sin\theta_p}\re(\zeta^\nu e^{-i\phi})\quad\text{with}\quad\varepsilon(p)=\zeta^2,
$$
for some $\theta_p\in(0,\pi)$. There is no loss of generality in assuming $\re(\zeta^\nu e^{-i\phi})>0$. Then the sign of the sequence $\re(a(p^\nu)e^{-i\phi})$ is the same as the sign of $\sin((\nu+1)\theta_p)$, when $p$ varies over primes satisfying $\varepsilon(p)=\zeta^2$. Thus we have
$$
\re(a(p^\nu)e^{-i\phi})>0\quad \text{if and only if}\quad \theta_p\in A_{>0},
$$
and
$$
\re(a(p^\nu)e^{-i\phi})<0 \quad \text{if and only if}\quad \theta_p\in A_{<0},
$$
where $A_{\lessgtr 0}$ be defined as in the proof of Lemma \ref{lem:2}. From Theorem \ref{thmCMST} and \ref{thmST}, we know that the sequence $\{\theta_p\}_{p,\varepsilon(p)=\zeta^2}$ is equidistributed in $[0,\pi]$ with respect to the measure $\mu$. Thereby, there are infinitely many primes $p$ satisfying $\varepsilon(p)=\zeta^2$ such that $\theta_p\in A_{>0}$, and infinitely many primes $p$ satisfying $\varepsilon(p)=\zeta^2$ such that $\theta_p\in A_{<0}$. Hence the sequence $\{\re(a(p^\nu)e^{-i\phi})\}_{p,\varepsilon(p)=\zeta^2}$ changes sign infinitely often.

\item[\textbf{Case 2}: $\mathrm{arg}(\zeta^\nu)\equiv \phi\pm\frac{\pi}{2}\pmod{2\pi}$.] The sequence $\{\re(a(p^\nu) e^{-i\phi})\}_{p,\varepsilon(p)=\zeta^2}$ is trivial.
\end{description}

Summarizing, we have thus proved that for any root of unity $\zeta$ such that $\zeta^2\in\mathrm{Im}(\varepsilon)$, the sequence $\{\re(a(p^\nu)e^{-i\phi})\}_{p}$ when $p$ runs over primes satisfying $\varepsilon(p)=\zeta^2$, either changes sign infinitely often or is trivial. Accordingly, for each $\phi\in [0,\pi)$ either the sequence $\{\re(a(p^\nu)e^{-i\phi})\}_{p\in\mathbb{P}}$ is trivial or changes sign infinitely often.

It remains to calculate the natural density of the sets $P_{>0}(\phi,\nu)$ and $P_{<0}(\phi,\nu)$. Here we restrict ourselves to the case when $f$ is not of CM-type, as the argument is entirely similar to the CM situation. The key point here is to see that 
$$P_{>0}(\phi,\nu)=\coprod_{\mycom{\xi,r_\varepsilon\text{-th root of unity}}{\xi=\zeta^2,\re(\zeta^\nu e^{-i\phi})>0}}\mathbb{P}_{>0}(\zeta,\nu)\bigsqcup \coprod_{\mycom{\xi,r_\varepsilon\text{-th root of unity}}{\xi=\zeta^2,\re(\zeta^\nu e^{-i\phi})<0}}\mathbb{P}_{<0}(\zeta,\nu),$$
and
$$P_{<0}(\phi,\nu)=\coprod_{\mycom{\xi,r_\varepsilon\text{-th root of unity}}{\xi=\zeta^2,\re(\zeta^\nu e^{-i\phi})<0}}\mathbb{P}_{>0}(\zeta,\nu)\bigsqcup \coprod_{\mycom{\xi,r_\varepsilon\text{-th root of unity}}{\xi=\zeta^2,\re(\zeta^\nu e^{-i\phi})>0}}\mathbb{P}_{<0}(\zeta,\nu),$$
up to finitely many primes, where $\zeta=e^{\frac{\pi i j}{r_{\varepsilon}}}$, is so chosen that $1\le j \le r_{\varepsilon}$, when $r_{\varepsilon}$ is even, (note that there is a unique choice of $\zeta$ when $r_{\varepsilon}$ is odd). 

The above displayed formula combined with Lemma \ref{lem:2}, gives
\begin{eqnarray*}
\delta\left(P_{>0}(\phi,\nu)\right) &=& \!\!\lim_{x\to\infty}\sum_{\mycom{\xi,r_\varepsilon\text{-th root of unity}}{ \xi=\zeta^2 ,\re(\zeta^\nu e^{-i\phi})>0}}\!\!\!\!\!\dfrac{\pi_{>0}(x,\zeta)}{\pi(x)}+\lim_{x\to\infty}\sum_{\mycom{\xi,r_\varepsilon\text{-th root of unity}}{\xi=\zeta^2 ,\re(\zeta^\nu e^{-i\phi})<0}}\!\!\!\!\!\!\!\dfrac{\pi_{<0}(x,\zeta)}{\pi(x)},\\
  &=& \!\!\frac{1}{2}\sum_{\mycom{\xi,r_\varepsilon\text{-th root of unity}}{ \xi=\zeta^2,\re(\zeta^\nu e^{-i\phi})\neq 0}}\frac{1}{r_{\varepsilon}},\\
  &=&\!\!\frac{\delta(P_{\neq 0}(\phi,\nu))}{2}.
\end{eqnarray*}
In the same manner we can see that $\delta\left(P_{<0}(\phi,\nu)\right) =\frac{\delta(P_{\neq 0}(\phi,\nu))}{2}$, which concludes the proof.
\end{proof}

\sloppy Our next concern will be the oscillatory behavior of the sequence $\{a(p^\nu)\}_{\nu\in\mathbb{N}}$, and the equidistribution of signs of $\{\re(a(p^\nu)e^{-i\phi})\}_{\nu\in\mathbb{N}}$, for a fixed prime number $p$. 
\begin{thm}\label{thm:2}
Let $f\in S_k^{\mathrm{new}}(N,\varepsilon)$ be a normalized newform, and let
$$
f(z)=\sum_{n\ge 1}a(n)n^{(k-1)/2}q^n\quad z\in\mathcal{H}
$$ 
be its Fourier expansion at $\infty$. Then there exists a set $S$ of primes of density zero, such that the following holds: For every prime $p\notin S$, the sequence $\{a(p^\nu)\}_{\nu\in\mathbb{N}}$ is oscillatory, and for any $\phi\in [0,\pi)$ we have 
$$ 
\lim_{x\to\infty}\dfrac{\#\{\nu\le x\; :\; \re\{a(p^\nu)e^{-i\phi}\}\gtrless 0\}}{\#\{\nu\le x \; :\; \re\{a(p^\nu)e^{-i\phi}\} \neq 0\}}=\frac{1}{2}.
$$
\end{thm}
\begin{rem}
\begin{enumerate}
\item The sequence $\{a(p^\nu)\}_{\nu\in\mathbb{N}}$ is oscillatory means that the sequence escape infinitely often from any half-plane. Hence, it improves the result of the author in \cite[Theorem 2.1]{Amri}.
\item It is worth pointing out that the theorem holds for prime in $S$, under some further restrictions, which can be easily deduced from the techniques of our proof. 
\end{enumerate}
\end{rem}
\begin{proof}[Proof of Theorem \ref{thm:2}] Set
$$
S :=\coprod_{\xi\in\mathrm{Im}(\varepsilon)}\left\{p\in\mathbb{P} : \varepsilon(p)=\xi, \theta_p=0\:\text{or}\:\pi\right\},
$$
where $\theta_p$ is defined as in \eqref{eq7}. By Theorem \ref{thmCMST} and \ref{thmST} we see that the set $S$ has density zero. Let $p$ be a prime outside $S$, then there exists a root of unity $\zeta$ satisfying $\zeta^2\in\mathrm{Im}(\varepsilon)$, such that $\varepsilon(p)=\zeta^2$. From Lemma \ref{lem:1}, one can write
$$
\re(a(p^\nu)e^{-i\phi})=\dfrac{\sin((\nu+1)\theta_p)}{\sin\theta_p}\re(\zeta^\nu e^{-i\phi}),
$$ 
for some $\theta_p\in(0,\pi)$. We derive the oscillatory behavior of the sequence $\{a(p^\nu))\}_{\nu\in\mathbb{N}}$, from the well-known behavior of the sequence $\left\{\sin((\nu+1)\theta_p)\right\}_{\nu\in\mathbb{N}}$.

We are left with the task of studying the equidistribution of signs of $\{\re(a(p^\nu)e^{-i\phi})\}_{\nu\in\mathbb{N}}$, to this end, write $\re(\zeta^\nu e^{-i\phi})=\cos\left(\frac{\pi j\nu}{r_{\varepsilon}}-\phi\right)$, if $r_{\varepsilon}$ is even, and $\re(\zeta^\nu e^{-i\phi})=\cos\left(\frac{2\pi j\nu}{r_{\varepsilon}}-\phi\right)$, if  $r_{\varepsilon}$ is odd, for some $1\le j\le r_{\varepsilon}$.  We treat only the former case, the second one being completely similar. We need to distinguish the following two cases.
\begin{description}
\item[\textbf{Case 1}: $\frac{\theta_p}{2\pi}$ is irrational.] Write
$$
\re(a(p^{\nu})e^{-i\phi})=(\sin\theta_p)^{-1}\sin\left(2\pi \left<\!\frac{(\nu+1)\theta_p}{2\pi}\!\right>\right)\cos\left(\frac{\pi j\nu}{r_{\varepsilon}}-\phi\right),
$$ 
where $\left<\!\frac{(\nu+1)\theta_p}{2\pi}\!\right>$ denotes the fractional part of $\frac{(\nu+1)\theta_p}{2\pi}$. Note that the sequence $\left\{\cos\left(\frac{\pi j\nu}{r_{\varepsilon}}-\phi\right)\right\}_{\nu\in\mathbb{N}}$ is $t_{\varepsilon}$-periodic and takes only finitely many different values, with $t_{\varepsilon}=2r'_{\varepsilon}$ if $\frac{j}{(r_{\varepsilon},j)}$ is odd, and $t_{\varepsilon}=r'_{\varepsilon}$ if $\frac{j}{(r_{\varepsilon},j)}$ is even, where $r'_\varepsilon=\frac{r_{\varepsilon}}{(r_{\varepsilon},j)}$. So, one may split the total range for $\nu$ into different arithmetic progressions $d\pmod{ t_{\varepsilon}}$ where $1\le d \le t_{\varepsilon}$, so that when $\nu$ runs through each of these arithmetic progressions the $\cos$-factor becomes constant and takes always the same value, say $c_d$. Accordingly we may write
$$
\#\{\nu\le x : \re(a(p^{\nu})e^{-i\phi})>0\}=\sum_{\mycom{d=1}{c_d>0}}^{t_{\varepsilon}}\sum_{\mycom{\nu\le x, \nu\equiv d \!\!\!\!\!\pmod {t_{\varepsilon}}}{\sin\left(2\pi\left<\frac{\nu \theta_p}{2\pi}\right>\right)>0}}1+ \sum_{\mycom{d=1}{c_d<0}}^{t_{\varepsilon}}\sum_{\mycom{\nu\le x, \nu\equiv d \!\!\!\!\!\pmod {t_{\varepsilon}}}{\sin\left(2\pi\left<\frac{\nu \theta_p}{2\pi}\right>\right)<0}}1,
$$
and
$$
\#\{\nu\le x : \re(a(p^{\nu})e^{-i\phi})<0\}=\sum_{\mycom{d=1}{c_d>0}}^{t_{\varepsilon}}\sum_{\mycom{\nu\le x, \nu\equiv d \!\!\!\!\!\pmod {t_{\varepsilon}}}{\sin\left(2\pi\left<\frac{\nu \theta_p}{2\pi}\right>\right)<0}}1+ \sum_{\mycom{d=1}{c_d<0}}^{t_{\varepsilon}}\sum_{\mycom{\nu\le x, \nu\equiv d \!\!\!\!\!\pmod {t_{\varepsilon}}}{\sin\left(2\pi\left<\frac{\nu \theta_p}{2\pi}\right>\right)>0}}1.
$$ 
On the other hand we have
\begin{eqnarray*}
\lim_{x\to\infty}\frac{1}{x}\sum_{\mycom{\nu\le x, \nu\equiv d \!\!\!\!\!\pmod {t_{\varepsilon}}}{\sin\left(2\pi\left<\frac{\nu \theta_p}{2\pi}\right>\right)>0}}1 &=& \lim_{x\to\infty}\frac{1}{x}\dsum_{\mycom{\nu\le x, \nu\equiv d \!\!\!\!\!\pmod {t_{\varepsilon}}}{\sin\left(2\pi\left<\frac{\nu \theta_p}{2\pi}\right>\right)\in [0,1]}} 1,\\
 &=&  \lim_{x\to\infty}\frac{1}{x}\dsum_{\mycom{\nu\le x, \nu\equiv d \!\!\!\!\!\pmod {t_{\varepsilon}}}{2\pi\left<\frac{\nu \theta_p}{2\pi}\right>\in \left[0,\pi\right]}} 1,\\
  &=& \frac{1}{2t_\varepsilon},
 \end{eqnarray*}
\sloppy where we used the fact that the sequence $\left\{\left<\frac{\nu \theta_p}{2\pi}\right>\right\}_{\nu\in\mathbb{N}}$ is uniformly distributed $\pmod 1$ in $[0,1]$ in the last step, (Weyl's equidistribution theorem, see, e.g., \cite{Kuipers}[Example 2.1, pp.8]). Similarly we have
$$
\lim_{x\to\infty}\frac{1}{x}\sum_{\mycom{\nu\le x, \nu\equiv d \!\!\!\!\!\pmod {t_{\varepsilon}}}{\sin\left(2\pi\left<\frac{\nu \theta_p}{2\pi}\right>\right)<0}}1=\frac{1}{2t_{\varepsilon}}.
$$
It follows
$$
\lim_{x\to\infty}\dfrac{\#\{\nu\le x : \re(a(p^{\nu})e^{-i\phi})\gtrless 0\}}{x}=\sum_{\mycom{d=1}{c_d\ne 0}}^{t_{\varepsilon}}\frac{1}{2 t_{\varepsilon}},
$$
and therefore 
$$
\lim_{x\to\infty}\dfrac{\#\{\nu\le x : \re(a(p^{\nu})e^{-i\phi})\gtrless 0\}}{\#\{\nu\le x : \re(a(p^{\nu})e^{-i\phi})\ne 0\}}=\frac{1}{2}.
$$
\item[\textbf{Case 2}: $\frac{\theta_p}{2\pi}=\frac{n}{m}\in (0,\frac{1}{2})$ is rational, where $m$ and $n$ are coprime.] Write
$$
\re(a(p^{\nu})e^{-i\phi})=(\sin\theta_p)^{-1}\sin\left(\frac{2\pi(\nu+1) n}{m}\right)\cos\left(\frac{\pi j\nu}{r_{\varepsilon}}-\phi\right).
$$
By similar considerations as in the previous case, we have
$$
\#\{\nu\le x : \re(a(p^{\nu})e^{-i\phi})>0\}=\sum_{\mycom{d=1}{c_d>0}}^{t_{\varepsilon}}\sum_{\mycom{\nu\le x, \nu\equiv d \!\!\!\!\!\pmod {t_{\varepsilon}}}{\sin\left(\frac{2\pi\nu n}{m}\right)>0}}1+ \sum_{\mycom{d=1}{c_d<0}}^{t_{\varepsilon}}\sum_{\mycom{\nu\le x, \nu\equiv d \!\!\!\!\!\pmod {t_{\varepsilon}}}{\sin\left(\frac{2\pi\nu n}{m}\right)<0}}1,
$$
and
$$
\#\{\nu\le x : \re(a(p^{\nu})e^{-i\phi})<0\}=\sum_{\mycom{d=1}{c_d>0}}^{t_{\varepsilon}}\sum_{\mycom{\nu\le x, \nu\equiv d \!\!\!\!\!\pmod {t_{\varepsilon}}}{\sin\left(\frac{2\pi\nu n}{m}\right)<0}}1+ \sum_{\mycom{d=1}{c_d<0}}^{t_{\varepsilon}}\sum_{\mycom{\nu\le x, \nu\equiv d \!\!\!\!\!\pmod {t_{\varepsilon}}}{\sin\left(\frac{2\pi\nu n}{m}\right)>0}}1.
$$ 
Thus, the study of the distribution of signs of $\re(a(p^{\nu})e^{-i\phi})$, turned out to study the distribution of signs of the sequence $\left\{\sin\left(\frac{2\pi\nu n}{m}\right)\right\}_{\nu}$ when $\nu$ runs through an arithmetic progression $d\pmod{t_\varepsilon}$. Note that this sequence is $t_{\varepsilon}m'$-periodic and takes only finitely many different values,  where $m'=\frac{m}{(m,t_\varepsilon)}$. Thus, we can split this arithmetic progression into $m'$ different sub-arithmetic progressions $d+t_{\varepsilon}\ell\pmod{m't_\varepsilon}$ where $1\le\ell\le m'$, such that when $\nu$ runs through each of these sub-arithmetic progressions the sequence becomes constant and has always the same value, say $s_{d,\ell}$. Consequently
\begin{eqnarray*}
\nonumber \lim_{x\to\infty}\frac{1}{x}\sum_{\mycom{\nu\le x, \nu\equiv d \!\!\!\!\!\pmod {t_{\varepsilon}}}{\sin\left(\frac{2\pi \nu n}{m}\right)>0}}1 &=& \lim_{x\to\infty}\frac{1}{x}\sum_{\mycom{\ell=1}{s_{d,\ell}>0}}^{m'}\sum_{\mycom{\nu\le x, \nu\equiv d \!\!\!\!\!\pmod {t_{\varepsilon}}}{\frac{\nu-d}{t_\varepsilon}\equiv \ell \!\!\!\!\!\pmod{m'}}}1, \\  
&=& \lim_{x\to\infty}\frac{1}{x}\sum_{\mycom{\ell=1}{s_{d,\ell}>0}}^{m'}\sum_{\mycom{\nu\le x}{\nu\equiv d+t_\varepsilon\ell \!\!\!\!\!\pmod {t_{\varepsilon}m'}}}1,\\ 
&=& \sum_{\mycom{\ell=1}{s_{d,\ell}>0}}^{m'}\frac{1}{t_{\varepsilon}m'}.
\end{eqnarray*}
Similarly, we obtain 
$$
\lim_{x\to\infty}\frac{1}{x}\sum_{\mycom{\nu\le x, \nu\equiv d \!\!\!\!\!\pmod {t_{\varepsilon}}}{\sin\left(\frac{2\pi \nu n}{m}\right)<0}}1=\sum_{\mycom{\ell=1}{s_{d,\ell}<0}}^{m'}\frac{1}{t_{\varepsilon}m'}.
$$
Therefore
$$
\lim_{x\to\infty}\dfrac{\#\{\nu\le x : \re(a(p^{\nu})e^{-i\phi})\gtrless 0\}}{x}=\sum_{\mycom{d=1}{c_d\ne 0}}^{t_{\varepsilon}}\sum_{\mycom{\ell=1}{s_{d,\ell}\ne 0}}^{m'}\frac{1}{2t_{\varepsilon}m'},
$$
and hence
$$
\lim_{x\to\infty}\dfrac{\#\{\nu\le x : \re(a(p^{\nu})e^{-i\phi})\gtrless 0\}}{\#\{\nu\le x : \re(a(p^{\nu})e^{-i\phi})\ne 0\}}=\frac{1}{2},
$$
as desired.
\end{description}
This completes the proof of Theorem \ref{thm:2}.
\end{proof}

\section{Equidistribution of sign results of half-integral weight cuspidals eigenforms}\label{sec:3}
In this section, we shall present and prove our results concerning the oscillatory behavior and signs equidistribution of Fourier coefficients of cuspidal Hecke eigenforms following a similar philosophy to that in the previous section. In order to state our results, we need to develop some notations and make some assumptions for this section. 

\begin{hyp}\label{hyp:2}
Let $N\ge 4$ be divisible by $4$ and $k\ge 1$ be a natural number. Fix any Dirichlet character $\varepsilon\pmod N$. Let $$f(z)=\sum_{n=1}^{\infty}a(n)q^n\quad z\in\mathcal{H},$$ be a non-zero cuspidal Hecke eigenform of half-integral weight $k+1/2$ and level $N$ with Dirichlet character $\varepsilon\pmod N$, and let  $t$ be a square-free integer such that $a(t)\ne 0$.  The Shimura correspondence \cite{Shi} lifts $f$ to a Hecke eigenform $\mathrm{Sh}_t(f)$ of weight $2k$ for the group $\Gamma_0(N/2)$ with character $\varepsilon^2$. Let us write
$$
\mathrm{Sh}_t(f)=\sum_{n\ge 1} A_t(n)q^n,
$$
for its expansion at $\infty$. For simplicity we assume that $a(t)=1.$ According to \cite{Shi}, the $n$-th Fourier coefficient of $\mathrm{Sh}_t(f)$ is given by 
\begin{equation}
A_t(n)=\sum_{d|n}\varepsilon_{t,N}(d)d^{k-1}a\left(\frac{n^2}{d^2}t\right),\label{eq8}
\end{equation}
where $\varepsilon_{t,N}$ denotes the character $\varepsilon_{t,N}(d):=\varepsilon(d)\left(\frac{(-1)^{k}N^{2}t}{d}\right)$, we let $\chi_0(d):=\left(\frac{(-1)^{k}N^{2}t}{d}\right)$.  If $\mathrm{Sh}_t(f)$ has complex multiplication by an imaginary quadratic field $F$ denote by $d_F$ its fundamental discriminant. We suppose that $(r_{\varepsilon},d_F)=1$ and the fields $F_\varepsilon$ and $\mathbb{Q}(\sqrt{(-1)^k t})$ are linearly disjoint over $\mathbb{Q}$, where $F_{\varepsilon}$ is the field obtained by adjoining to $F$ the values of $\varepsilon$. We let $\chi_F$ be the quadratic character associated to $F$.
\end{hyp}

Let $\zeta$ be a root of unity such that $\zeta\in\mathrm{Im}(\varepsilon)$, if $p$ is a prime number satisfying $\varepsilon(p)=\zeta,$ then we have
$$B_\zeta(p):=\frac{A_t(p)}{2p^{k-1/2}\zeta}\in [-1,1].$$
By \eqref{eq8} we have $a(tp^2)=A_t(p)-\varepsilon_{t,N}(p)p^{k-1}$, and hence 
\begin{equation}\label{eq9}
\frac{a(tp^2)}{2p^{k-1/2}\zeta}=B_{\zeta}(p)-\frac{\chi_0(p)}{2\sqrt{p}}.
\end{equation}

For abbreviation  we let $A_\zeta(p)$ stand for $\frac{a(tp^2)}{2p^{k-1/2}\zeta}$.
Set
$$P_{> 0}(\phi):=\{p\in\mathbb{P}\; :\; \re(a(tp^2)e^{-i\phi})> 0\},$$

$$P_{< 0}(\phi):=\{p\in\mathbb{P}\; :\; \re(a(tp^2)e^{-i\phi})< 0\},$$
and 
$$P_{\neq 0}(\phi):=\{p\in\mathbb{P}\; :\; \re(a(tp^2)e^{-i\phi})\neq 0\},$$
where $\phi$ is a real number belonging to $[0,\pi)$.
\begin{thm}\label{thm:3}
Let $f\in S_{k+1/2}(N,\varepsilon)$ be a cuspidal Hecke eigenform  satisfying Hypothesis \ref{hyp:2}. Let us write 
$$
f(z)=\sum_{n\ge 1}a(n)q^n\quad z\in\mathcal{H},
$$
for its Fourier expansion at $\infty$. Then the sequence $\{a(tp^2)\}_{p\in\mathbb{P}}$ is oscillatory, and if moreover $\mathrm{Sh}_t(f)$ is not of CM-type, or of CM-type and $\chi_0\ne \chi_{\mathrm{triv}},\chi_F$, then for each $\phi\in[0,\pi)$ the sets $P_{>0}(\phi)$ and $P_{<0}(\phi)$ have equal positive natural density, that is, both are precisely half of the natural density of the set $P_{\ne 0}(\phi)$. In the remaining cases we have
$$
\delta(P_{>0}(\phi))=\left\{
    \begin{array}{llll}
        \dfrac{\delta(P_{\neq0}(\phi))}{4}+\dfrac{1}{2}\dsum_{\mycom{\zeta\in\mathrm{Im}(\varepsilon)}{\re(\zeta e^{-i\phi})>0}} \frac{1}{r_{\varepsilon}} &\mbox{if}\;\;\chi_0 =\chi_F, \\
        \dfrac{\delta(P_{\neq0}(\phi))}{4}+ \dfrac{1}{2}\dsum_{\mycom{\zeta\in\mathrm{Im}(\varepsilon)}{\re(\zeta e^{-i\phi})<0}} \frac{1}{r_{\varepsilon}} & \mbox{if}\;\;  \chi_0 =\chi_{\mathrm{triv}},
    \end{array}
\right.
$$
and
$$
\delta(P_{<0}(\phi))=\left\{
    \begin{array}{llll}
        \dfrac{\delta(P_{\neq0}(\phi))}{4}+\dfrac{1}{2}\dsum_{\mycom{\zeta\in\mathrm{Im}(\varepsilon)}{\re(\zeta e^{-i\phi})<0}} \frac{1}{r_{\varepsilon}} &\mbox{if}\;\;\chi_0 =\chi_F, \\
        \dfrac{\delta(P_{\neq0}(\phi))}{4}+ \dfrac{1}{2}\dsum_{\mycom{\zeta\in\mathrm{Im}(\varepsilon)}{\re(\zeta e^{-i\phi})>0}} \frac{1}{r_{\varepsilon}} & \mbox{if}\;\;  \chi_0 =\chi_{\mathrm{triv}}.
    \end{array}
\right.
$$
\end{thm}

We shall need the following lemma.
\begin{lem}\label{lem:3}
We make the same assumptions as in Theorem \ref{thm:3}, and let $\zeta$ be a root of unity belonging to $\mathrm{Im}(\varepsilon)$. If we denote
 
$$\mathbb{P'}_{\gtrless0}(\zeta):=\left\{p\in\mathbb{P}\;:\; \varepsilon(p)=\zeta,\;A_\zeta(p)\gtrless0\right\}.$$
Then we have
$$
\delta(\mathbb{P'}_{>0}(\zeta))=\left\{
    \begin{array}{ll}
        \frac{1}{2r_{\varepsilon}} & \mbox{if}\;\; \mathrm{Sh}_t(f)\;\;  \mbox{is not of CM-type}, \\
        \frac{1}{2r_{\varepsilon}} & \mbox{if}\;\; \mathrm{Sh}_t(f)\;\;  \mbox{is of CM-type and}\;\; \chi_0\ne \chi_{\mathrm{triv}},\chi_F,\\
        \frac{1}{4r_{\varepsilon}} & \mbox{if}\;\; \mathrm{Sh}_t(f)\;\;  \mbox{is of CM-type and}\;\; \chi_0 =\chi_{\mathrm{triv}},\\
        \frac{3}{4r_{\varepsilon}} & \mbox{if}\;\; \mathrm{Sh}_t(f)\;\;  \mbox{is of CM-type and}\;\; \chi_0 =\chi_F,
    \end{array}
\right.
$$
and
$$
\delta(\mathbb{P'}_{<0}(\zeta))=\left\{
    \begin{array}{ll}
        \frac{1}{2r_{\varepsilon}} & \mbox{if}\;\; \mathrm{Sh}_t(f)\;\;  \mbox{is not of CM-type}, \\
        \frac{1}{2r_{\varepsilon}} & \mbox{if}\;\; \mathrm{Sh}_t(f)\;\;  \mbox{is of CM-type and}\;\; \chi_0\ne \chi_{\mathrm{triv}},\chi_F,\\
        \frac{3}{4r_{\varepsilon}} & \mbox{if}\;\; \mathrm{Sh}_t(f)\;\;  \mbox{is of CM-type and}\;\; \chi_0 =\chi_{\mathrm{triv}},\\
        \frac{1}{4r_{\varepsilon}} & \mbox{if}\;\; \mathrm{Sh}_t(f)\;\;  \mbox{is of CM-type and}\;\; \chi_0 =\chi_F.
    \end{array}
\right.
$$
\end{lem}

\begin{proof}
Denote by $\pi_{>0}(x,\zeta):=\#\{p\le x\; :\;p\in\mathbb{P'}_{>0}(\zeta)\}$ and similarly $\pi_{\ge 0}(x,\zeta)$, $\pi_{<0}(x,\zeta)$ and $\pi_{\le0}(x,\zeta)$. 

First assume that $\mathrm{Sh}_t(f)$ is not of CM-type, we follow closely the method of \cite{IW}. Let $p$ be a prime satisfying $\varepsilon(p)=\zeta$, from \eqref{eq9} we have 
$$
A_\zeta(p)>0 \Longleftrightarrow B_\zeta(p)>\frac{\chi_0(p)}{2\sqrt{p}}.
$$
It follows that for any fixed $\epsilon >0$, we have the following inclusion of sets
$$\{p\leq x : \varepsilon(p)=\zeta, B_\zeta(p)>\epsilon\}\!\subset\!\{p\in\mathbb P : p\leq\frac{1}{4\epsilon^2}, \varepsilon(p)=\zeta\}\cup\{p\leq x : p\in\mathbb{P'}_{>0}(\zeta)\}.$$
Therefore 
\begin{equation}\label{eq10}
\pi_{>0}(x,\zeta)+\pi_{\zeta}\left(\frac{1}{4\epsilon^2}\right)\geq \#\{p\leq x :\varepsilon(p)=\zeta,\;\; B_\zeta(p)>\epsilon\},
\end{equation}
where $\pi_{\zeta}(x):=\#\{p\in\mathbb{P}: p\le x,\;\; \varepsilon(p)=\zeta\}$. Now dividing \eqref{eq10} by $\pi_{\zeta}(x)$ we obtain 
\begin{equation}\label{eq:10}
\frac{\pi_{>0}(x,\zeta)}{\pi_{\zeta}(x)}+\dfrac{\pi_{\zeta}\left(\frac{1}{4\epsilon^2}\right)}{\pi_{\zeta}(x)}\geq \dfrac{\#\{p\leq x :\varepsilon(p)=\zeta,\;\; B_\zeta(p)>\epsilon\}}{\pi_{\zeta}(x)}.
\end{equation}
Since $\pi_{\zeta}(x)\underset{x\to\infty}{\sim}\frac{x}{r_{\varepsilon}\log x}$, and the term $\pi_{\zeta}\left(\frac{1}{4\epsilon^2}\right)$ is finite, it follows
\begin{equation}\label{eq11}
\lim_{x\to\infty}\dfrac{\pi_{\zeta}\left(\frac{1}{4\epsilon^2}\right)}{\pi_{\zeta}(x)}=0.
\end{equation}
On the other hand, by Theorem \ref{thmST} we have 
\begin{equation}\label{eq12}
 \lim_{x\to\infty}\dfrac{\#\{p\leq x : \varepsilon(p)=\zeta,\;\; B_\zeta(p)>\epsilon\}}{\pi_{\zeta}(x)}=\mu_{\mathrm{ST}}([\epsilon,1]).
 \end{equation}

Taking into account \eqref{eq11} and \eqref{eq12} a passage to the limit in \eqref{eq:10} implies that
\begin{equation}\label{eq13}
\liminf_{x\to\infty}\frac{\pi_{>0}(x,\zeta)}{\pi_{\zeta}(x)}\geq \mu_{\mathrm{ST}}([\epsilon,1]).
\end{equation}
As the inequality \eqref{eq13} holds for all $\epsilon>0$, we have
$$\liminf_{x\to\infty}\frac{\pi_{>0}(x,\zeta)}{\pi_{\zeta}(x)}\geq \frac{1}{2}.$$
Similarly we get $\liminf\limits_{x\to\infty}\frac{\pi_{\le 0}(x,\zeta)}{\pi_{\zeta}(x)}\geq \frac{1}{2}$ and in view of $\pi_{\le 0}(x,\zeta)=\pi_\zeta(x)-\pi_{>0}(x,\zeta)$, one sees $\limsup\limits_{x\to\infty}\frac{\pi_{>0}(x,\zeta)}{\pi_{\zeta}(x)}\le\frac{1}{2}$. Consequently
$$
\lim\limits_{x\to\infty}\frac{\pi_{>0}(x,\zeta)}{\pi_{\zeta}(x)}=\frac{1}{2}.
$$
Since $\delta(\{p\in\mathbb{P} : \varepsilon(p)=\zeta\})=\frac{1}{r_{\varepsilon}}$, it follows
 $$\delta(\mathbb{P}'_{>0}(\zeta))=\lim_{x\to\infty}\frac{\pi_{>0}(x,\zeta)}{\pi(x)}=\frac{1}{2r_{\varepsilon}}.$$
Similarly we have $\delta(\mathbb{P}'_{<0}(\zeta))=\frac{1}{2r_{\varepsilon}}$.

Now, let us examine the CM situation, assume that $\mathrm{Sh}_t(f)$ has CM by an imaginary quadratic field $F$. Set $I=(0,1]$, $J=[-1,0)$,
$$
T_{I}(\zeta):=\{p\in\mathbb{P}: \varepsilon(p)=
\zeta, A_{\zeta}(p)\in I, B_\zeta(p)\ne 0\}, S_{I}(\zeta):=\{p\in\mathbb{P} : \varepsilon(p)=\zeta, B_{\zeta}(p)\in I\}
$$
and 
$$
 T_{J}(\zeta):=\{p\in\mathbb{P}: \varepsilon(p)=\zeta, A_{\zeta}(p)\in J, B_\zeta(p)\ne 0\}, S_{J}(\zeta):=\{p\in\mathbb{P} : \varepsilon(p)=\zeta, B_{\zeta}(p)\in J\}
$$

From Theorem \ref{thmCMST} we have $\delta(S_{I}(\zeta))=\delta(S_{J}(\zeta))=\frac{1}{4 r_\varepsilon}$. Thus in view of \cite[Theorem 4.2.1]{Arias} and Theorem \ref{thmCMST} it follows that $$\delta(T_{I}(\zeta))=\delta(T_{J}(\zeta))=\frac{1}{4 r_\varepsilon}.$$
Now by \cite[Remark 4.2.2]{Arias} we may write

$$
\mathbb{P'}_{>0}(\zeta)= T_{I}(\zeta)\sqcup 
$$
\begin{equation}\label{eq15}
\left(\{p\in\mathbb{P}: \varepsilon(p)=\zeta,B_\zeta(p)=0 \}\cap \left\{p\in\mathbb{P} : \varepsilon(p)=\zeta, \frac{-\chi_0(p)}{2\sqrt{p}}\in I\right\}\right),
\end{equation}
and 
$$
\mathbb{P'}_{<0}(\zeta)=T_{J}(\zeta)\sqcup
$$
\begin{equation}\label{eq16}
\left(\{p\in\mathbb{P}: \varepsilon(p)=\zeta,B_\zeta(p)=0 \}\cap \left\{p\in\mathbb{P} : \varepsilon(p)=\zeta, \frac{-\chi_0(p)}{2\sqrt{p}}\in J\right\}\right).
\end{equation}

In order to calculate $\delta(\mathbb{P'}_{<0}(\zeta))$ and $\delta(\mathbb{P'}_{>0}(\zeta))$, we distinguish the following three cases
\begin{description}
\item[\textbf{Case 1}: $\chi_0\ne \chi_{\mathrm{triv}}$ and $\chi_F$.] By our hypothesis we have $(r_{\varepsilon},d_F)=1$, and the fields $\mathbb{Q}(\sqrt{(-1)^k t})$ and $F_\varepsilon$ are linearly disjoint over $\mathbb{Q}$. Hence by Chebotarev's theorem the intersections in \eqref{eq15} and \eqref{eq16} have natural density $\frac{1}{4r_{\varepsilon}}$ and $\frac{1}{4r_{\varepsilon}}$ respectively. Consequently $\delta(\mathbb{P'}_{>0}(\zeta))=\frac{1}{2r_{\varepsilon}}$ and $\delta(\mathbb{P'}_{<0}(\zeta))=\frac{1}{2r_{\varepsilon}}$.

\item[\textbf{Case 2}: $\chi_0=\chi_{\mathrm{triv}}$.] Since $(r_{\varepsilon},d_F)=1$ the intersections in \eqref{eq15} and \eqref{eq16} have natural density $0$ and $\frac{1}{2r_{\varepsilon}}$ respectively, it follows that $\delta(\mathbb{P'}_{>0}(\zeta))=\frac{1}{4r_{\varepsilon}}$ and 
$\delta(\mathbb{P'}_{<0}(\zeta))=\frac{3}{4r_{\varepsilon}}$.

\item[\textbf{Case 3}: $\chi_0=\chi_F$.] By Hypothesis \ref{hyp:2} we have $(r_{\varepsilon},d_F)=1$. Thus, the intersections in \eqref{eq15} and \eqref{eq16} have natural density $\frac{1}{2r_{\varepsilon}}$ and $0$ respectively. It follows that $\delta(\mathbb{P'}_{>0}(\zeta))=\frac{3}{4r_{\varepsilon}}$ and $\delta(\mathbb{P'}_{<0}(\zeta))=\frac{1}{4r_{\varepsilon}}$.
\end{description}
which finishes the proof of Lemma \ref{lem:3}.
\end{proof}

We proceed now to prove Theorem \ref{thm:3}.
\begin{proof}[Proof of Theorem \ref{thm:3}]
Fix $\phi\in[0,\pi)$, and pick $\zeta$ a root of unity belonging to $\mathrm{Im}(\varepsilon)$. We first study the oscillatory behavior of the sequence $\{a(tp^2)\}_{p\in\mathbb{P}}$. To this end, two cases we shall need to consider.
\begin{description}

\item[\textbf{Case 1}: $\mathrm{arg}(\zeta)\not\equiv  \phi\pm\frac{\pi}{2}\pmod{2\pi}$.] The sequence $\{\re(a(tp^2)e^{-i\phi})\}_{p,\varepsilon(p)=\zeta}$ is not trivial, and we may write
$$
\re(a(tp^2)e^{-i\phi})=\frac{a(tp^2)}{\zeta}\re(\zeta e^{-i\phi}).
$$
By Theorem \ref{thmCMST} and \ref{thmST}, it follows that the sequence $\left(\frac{a(tp^2)}{\zeta}\right)_{p,\varepsilon(p)=\zeta}$
changes sign infinitely often. Hence the sequence $\{\re(a(tp^2)e^{-i\phi})\}_{p}$ changes sign infinitely often as $p$ varies over primes satisfying $\varepsilon(p)=\zeta$.
\item[\textbf{Case 2}: $\mathrm{arg}(\zeta)\equiv  \phi\pm\frac{\pi}{2}\pmod{2\pi}$.] The sequence $\{\re(a(tp^2)e^{-i\phi})\}_p$ is trivial when $p$ runs over primes satisfying $\varepsilon(p)=\zeta$.
\end{description}

What we have just proved is that for any root of unity $\zeta$ such that $\zeta\in\mathrm{Im}(\varepsilon)$, the sequence $\{\re(a(tp^2)e^{-i\phi})\}_{p}$ when $p$ runs over primes satisfying $\varepsilon(p)=\zeta$  either changes sign infinitely often or is trivial. Consequently, for each $\phi\in [0,\pi)$ either the sequence $\{\re(a(tp^2)e^{-i\phi})\}_{p\in\mathbb{P}}$
is trivial or changes sign infinitely often. 

For the purpose to calculate the density of the sets $P_{>0}(\phi)$ and $P_{<0}(\phi)$, we shall restrict ourselves to the case when $\mathrm{Sh}_t(f)$ is not of CM-type, as the argument in the CM situation is entirely analogous. First note that
$$P_{> 0}(\phi)=\coprod_{\mycom{\zeta,\text{root of unity}}{ \zeta\in\text{Im}(\varepsilon) ,\re(\zeta e^{-i\phi})>0}}\mathbb{P'}_{>0}(\zeta)\bigsqcup\coprod_{\mycom{\zeta,\text{root of unity}}{ \zeta\in\text{Im}(\varepsilon) ,\re(\zeta e^{-i\phi})<0}}\mathbb{P'}_{<0}(\zeta),$$
and
$$P_{< 0}(\phi)=\coprod_{\mycom{\zeta,\text{root of unity}}{\zeta\in\text{Im}(\varepsilon) ,\re(\zeta e^{-i\phi})<0}}\mathbb{P'}_{>0}(\zeta)\bigsqcup\coprod_{\mycom{\zeta,\text{root of unity}}{\zeta\in\text{Im}(\varepsilon) ,\re(\zeta e^{-i\phi})>0}}\mathbb{P'}_{<0}(\zeta),$$
up to finitely many primes. 

These, together with Lemma \ref{lem:3} yields
\begin{eqnarray*}
\delta\left(P_{>0}(\phi)\right) &=& \lim_{x\to\infty}\sum_{\mycom{\zeta,\text{root of unity}}{\zeta\in\text{Im}(\varepsilon) ,\re(\zeta e^{-i\phi})>0}}\dfrac{\pi_{>0}(x,\zeta)}{\pi(x)}+\lim_{x\to\infty}\sum_{\mycom{\zeta,\text{root of unity}}{\zeta\in\text{Im}(\varepsilon) ,\re(\zeta e^{-i\phi})<0}}\dfrac{\pi_{<0}(x,\zeta)}{\pi(x)},\\
  &=& \frac{1}{2}\sum_{\mycom{\zeta,\text{root of unity}}{\zeta\in\mathrm{Im}(\varepsilon) ,\re(\zeta e^{-i\phi})\neq 0}}\frac{1}{r_{\varepsilon}},\\
  &=&\frac{\delta(P_{\neq 0}(\phi))}{2}.
\end{eqnarray*}
In a similar way, it can be shown that $\delta\left(P_{<0}(\phi)\right) =\frac{\delta(P_{\neq 0}(\phi))}{2}$, which concludes the proof.
\end{proof}

Our next objective is to investigate the oscillatory behavior of the sequence $\{a(tp^{2\nu})\}_{\nu\in\mathbb{N}}$. We shall prove the following.
\begin{thm}\label{thm:4}
Let $f\in S_{k+1/2}(N,\varepsilon)$ be a cuspidal Hecke eigenform of half integral weight, and 
$$
f(z)=\sum_{n\ge 1}a(n)q^n\quad z\in \mathcal{H}
$$
its expansion at $\infty$. Let $t$ be a square-free integer such that $a(t)\ne 0$. For all but finitely many  primes $p$ the sequence $\{a(tp^{2\nu})\}_{\nu\in\mathbb{N}}$ is oscillatory.
\end{thm}
\begin{proof}
Applying the M\"obius inversion formula to \eqref{eq8}, we derive that
$$
a(tn^2)=\sum_{d |n} \mu(d)\varepsilon_{t,N}(d)d^{k-1} A_t\left(\frac{n}{d}\right).
$$
For $n=p^{\nu}$, with $\nu\in\mathbb{N}$ ($p\nmid N$ a prime), it follows that
\begin{equation}\label{eq18}
a(tp^{2\nu})=A_t(p^{\nu})-p^{k-1}\varepsilon_{t,N}(p) A_t(p^{\nu-1})
\end{equation}
Dividing \eqref{eq18} by $\varepsilon(p)^\nu$, we obtain
$$
\dfrac{a(tp^{2\nu})}{\varepsilon(p)^\nu}=\dfrac{A_t(p^{\nu})}{\varepsilon(p)^\nu}-\chi_0(p)p^{k-1}\dfrac{A_t(p^{\nu-1})}{\varepsilon(p)^{\nu-1}},
$$
hence $\frac{a(tp^{2\nu})}{\varepsilon(p)^\nu}\in\mathbb{R}$. Thus, we may write
$$
\re(a(tp^{2\nu})e^{-i\phi})=\frac{a(tp^{2\nu})}{\varepsilon(p)^\nu}\re(\varepsilon(p)^{\nu}e^{-i\phi}),
$$
for each $\phi\in[0,\pi)$. We shall have established the theorem if we prove that  for all but finitely many primes $p$ the sequence $\left(\frac{a(tp^{2\nu})}{\varepsilon(p)^\nu}\right)_{\nu\in\mathbb{N}}$, changes sign infinitely often. To this end, we shall follow \cite[Proof of Theorem 2.2]{bruinier}.

Assume, for the sake of contradiction that there exist infinitely many primes $p$ such that the sequence $\left(\frac{a(tp^{2\nu})}{\varepsilon(p)^\nu}\right)_{\nu\in\mathbb{N}}$ does not changes sign infinitely often. Let $\lambda_p$ denote the $p$-th Hecke eigenvalue of $f$. Since 
$$
T(p)\mathrm{Sh}_t(f)=\mathrm{Sh}_t(T(p^2)f),
$$
it follows that the $p$-th Hecke eigenvalue of $\mathrm{Sh}_t(f)$ is $\lambda_p$, where $T(p^2)$ is the Hecke operator on $S_{k+1/2}(N,\varepsilon)$ and $T(p)$ is the Hecke operator
on $S_{2k}(N/2,\varepsilon^2)$. By \cite[Corolary 1.8]{Shi} we have 
\begin{equation}\label{eq19}
\sum_{\nu\ge 0}a(tp^{2\nu})X^\nu=a(t)\dfrac{1-\varepsilon_{N,t}(p)p^{k-1}X}{1-\lambda_pX+\varepsilon(p)^2p^{2k-1}X^2},
\end{equation}
write 
$$1-\lambda_pX+\varepsilon(p)^2p^{2k-1}X^2=(1-\alpha_p X)(1-\beta_p X).$$ 
Replacing $X=\varepsilon(p)^{-1}p^{-s}$ ($s\in\mathbb{C}$) in \eqref{eq19} we get
\begin{equation}\label{eq20}
\sum_{\nu\ge0} a(tp^{2\nu})\varepsilon(p)^{-\nu}p^{-s\nu}=a(t)\dfrac{1-\chi_{0}(p)p^{k-1-s}}{(1-\alpha_{p}' p^{-s})(1-\beta_{p}'p^{-s})},
\end{equation}
where  $\alpha_{p}'=\alpha_p \varepsilon(p)^{-1}$ and $\beta_{p}'=\beta_p \varepsilon(p)^{-1}$. 

 Let $p$ be a prime such that $\frac{a(tp^{2\nu})}{\varepsilon(p)^\nu}\ge 0$ for all but finitely many $\nu \ge 0$. Thus, by Landau's theorem \cite[ pp. 697--699]{Landau}, the series in the left-hand side of \eqref{eq20} either (a) converges for all $s\in\mathbb{C}$ or (b) has a singularity at the real point of its line of convergence. It is clear that the alternative (a) cannot occur, since the right-hand side of \eqref{eq20} has a pole for $p^{s}=\alpha'_p$ or $p^{s}=\beta'_p.$
Thus the alternative (b) must hold, therefore $\alpha'_p$ or $\beta'_p$ must be real. On the other hand since
 $\frac{\lambda_p}{\varepsilon(p)}\in\mathbb{R}$, it follows $\alpha_p'=\overline{\beta_p'}$. Moreover, by Deligne's theorem \cite[Theorem 8.2]{Deligne}
we have $|\alpha'_p|=|\beta'_p|=p^{k-1/2}$. Consequently
$$\lambda_p=\pm 2 p^{k-1/2}\varepsilon(p),$$
hence $\mathbb{Q}(\sqrt{p})\subset K_f$, where $K_f$ denotes the field generated   over $\mathbb{Q}$ by $\lambda_p$ ($p$ runs over primes numbers) and all the values of $\varepsilon$. Therefore, by our hypothesis $K_f$ has infinitely many quadratic subfields, this is in contradiction with the fact that $K_f$ is a number field. Consequently, for all but finitely many primes $p$ and each $\phi\in [0,\pi)$ the sequence $\{\re(a(tp^{2\nu})e^{-i\phi}\}_{\nu\in\mathbb{N}}$ , changes sign infinitely often. 
\end{proof}
\begin{rem}
It seems likely that we can prove similar results to \cite[Theorem 3]{Kohnen} for the sequence $\{\re(a(tp^{2\nu})e^{-i\phi})\}_{\nu\in\mathbb{N}}$. However, we have not checked this as yet.
\end{rem}
\section{Concluding Remarks}
Let $k,N$ be natural numbers and $\varepsilon$ be a Dirichlet character modulo $N$ assume that $N$ be an odd and square-free integer. We write $S^{+}_{k+1/2}(4N,\varepsilon)$ for the Kohnen's plus space (cf. \cite{Kohnen1982}). 

Let $f$ be a cusp form of half integral weight belonging to $S^{+}_{k+1/2}(4N,\varepsilon)$, we let $a(n)$ to denote its $n$-th Fourier coefficient. Motivated by Theorem \ref{thm:3} and numerical calculations, it seems reasonable to conjecture that, for each $\phi\in [0,\pi)$ 
$$ 
\lim_{x\to\infty}\dfrac{\#\{n\le x\; :\; \im\{a(n)e^{-i\phi}\}\gtrless 0\}}{\#\{n\le x \; :\; \im\{a(n)e^{-i\phi}\} \neq 0\}}=\frac{1}{2},
$$

$$ 
\lim_{x\to\infty}\dfrac{\#\{n\le x\; :\; \re\{a(n)e^{-i\phi}\}\gtrless 0\}}{\#\{n\le x \; :\; \re\{a(n)e^{-i\phi}\} \neq 0\}}=\frac{1}{2}.
$$

It may be noted that one can get similar statements on the imaginary part in Theorem \ref{thm:1}, \ref{thm:2} and \ref{thm:3} by a rotation around $\pi/2$.

We believe that these results (Theorem \ref{thm:1}, \ref{thm:2}, \ref{thm:3} and the above conjecture) should extend to any totally real number field by the approach taken in the present paper.

\section*{Acknowledgments}
The author is greatly grateful to Francesc Fit\'e for a helpful conversation. He also wishes to thank Gabor Wiese for his valuable comments on the first draft of this work as well as Ilker Inam for providing him with some data for numerical experiments. Thanks are also due to the referee for his careful reading and their helpful comments which improve the paper.

\end{document}